\documentclass[11pt,a4paper]{amsart}
\usepackage[utf8]{inputenc}

\usepackage{tikz}
\usepackage{hyperref}
\usepackage{mathtools}

\usepackage[
  backend=biber,
  style=numeric,
  maxnames=5,
  maxcitenames=5,
  sorting=nyt,
  giveninits=true,
  doi=false,
  isbn=false,
  url=false,
  eprint=true
]{biblatex}
\usepackage{mymacro}

\newcommand{\pvar}[1]{#1\text{-}\mathrm{var}}
\newcommand{\mesh}[1]{| #1 |}
\newcommand{\osc}{\operatorname{osc}}

\begin{document}
\emergencystretch=2em

\title{Liouville regular Jordan curves}

\author{Tomohiro Asano}
\address{Department of Mathematics, Kyoto University, Kitashirakawa-Oiwake-Cho, Sakyo-ku, 606-8502, Kyoto, Japan}
\email{tasano@math.kyoto-u.ac.jp}
\thanks{The first author was supported in part by JSPS KAKENHI Grant Number JP24K16920 and partially by JST, CREST Grant Number JPMJCR24Q1, Japan.}

\author{Yuichi Ike}
\address{Graduate School of Mathematical Sciences, The University of Tokyo, 3-8-1 Komaba, Meguro-ku, Tokyo 153-8914, Japan}
\email{ike@ms.u-tokyo.ac.jp}
\thanks{The second author was supported in part by JSPS KAKENHI Grant Numbers JP22H05107 and 25K17254, and partially by JST, CREST Grant Number JPMJCR24Q1, Japan.}

\date{}

\begin{abstract}
    We investigate Liouville regular Jordan curves, which were introduced in our previous paper as ``Jordan curves that admit continuous Legendrian lifts''.
    We prove that a Jordan curve satisfying certain variation conditions in a neighborhood of each point is Liouville regular. 
    We also prove that the boundary of a Jordan domain is Liouville regular if the images of the radial tails under a conformal mapping have uniformly vanishing length. 
    Together with our previous result, these criteria show that each such Jordan curve admits an inscribed rectangle in every prescribed similarity class.
\end{abstract}

\maketitle

\section{Introduction}

The Square Peg Problem, first posed by Toeplitz~\cite{Toeplitz} in 1911, asks whether every continuous Jordan curve in the plane inscribes a square. 
The Rectangular Peg Problem generalizes this question and asks whether a Jordan curve inscribes a rectangle of a prescribed similarity class.
In \cite{GL21}, Greene and Lobb proved that for every rectangle $R$ in $\bR^2$, every smooth Jordan curve inscribes a rectangle similar to $R$.
In our previous paper~\cite{AI_peg},  we proved the existence of inscribed rectangles for a certain class of Jordan curves.
This result was independently obtained by St\'ephane Guillermou.

\begin{theorem}[{\cite[Thm.~1.3]{AI_peg}}]\label{thm:peg}
    Let $R$ be a rectangle in $\bR^2$.  
    Every Liouville regular Jordan curve inscribes a rectangle similar to $R$.
\end{theorem}

Here, Liouville regularity in the theorem is defined as follows. 
Write $(x,\xi)$ for the coordinates on $\bR^2$ and $\lambda=\xi\,dx$ for the Liouville $1$-form.
Let $e\colon\bR\to\bR/2\pi\bZ\simeq S^1$ denote the quotient map.

\begin{definition}[{\cite[Def.~1.1]{AI_peg}}]\label{def:curve}
    A Jordan curve $c\colon S^1\to\bR^2$ is said to be \emph{Liouville regular} if there exists a sequence of $C^1$-embeddings $c_n\colon S^1\to\bR^2$ such that
    \begin{enumerate}[(1)]
    \item $c_n\to c$ uniformly;
    \item for the functions $f_n \colon \bR \to \bR$ defined by 
    \[
        f_n(t) \coloneqq \int_0^t(c_n\circ e)^*\lambda,
    \]
    there exists a continuous function $f \colon\bR\to\bR$ such that $f_n \to f$ uniformly on every compact subset of $\bR$.
    \end{enumerate}
    We call $f$ a \emph{Liouville primitive} for $c$.
    A sequence $(c_n,f_n)_n$ satisfying the above conditions is called an \emph{approximating sequence} for $(c,f)$.
\end{definition}

\begin{remark}\label{rem:naming}
    In our previous paper~\cite{AI_peg}, the approximating embeddings were taken to be smooth, and the property in \cref{def:curve} was expressed by saying that the Jordan curve \emph{admits a continuous Legendrian lift}.
    This gives the same class of curves: a $C^1$-embedding can be approximated in $C^1$ by smooth embeddings, and the corresponding primitives converge uniformly, by a diagonal argument that is identical to the one used in the proof of \cref{lem:diagonal}.
    We therefore use the shorter term ``Liouville regular'' and work with $C^1$-approximations throughout.
\end{remark}

\begin{remark}[Legendrian interpretation]\label{rem:legendrian}
    Let $c \colon S^1 \to \bR^2$ be a Liouville regular Jordan curve with Liouville primitive $f$ and $(c_n,f_n)_n$ be an approximating sequence for $(c,f)$.
    We equip $\bR^3_{x,\xi,z}$ with the contact form $dz-\xi\,dx$.  
    Then the curve $\widehat c_n$ defined by
    $\widehat c_n(t)=(c_n(t),f_n(t))$ is a Legendrian curve 
    and the sequence $(\widehat c_n(t))_n$ converges to a curve $\widehat c$ defined by 
    \[\widehat c(t)=(c(t),f(t)). \]
    Rigidity phenomena of Legendrian submanifolds are known in many situations. 
    In particular, Dimitroglou Rizell--Sullivan~\cite{DRSknots} proved a $C^0$-rigidity result for Legendrian knots. 
    Hence, it is reasonable to regard $\widehat c$ as a ``continuous Legendrian embedding''. 
    This explains the terminology used in~\cite{AI_peg}.
\end{remark}

The integral of $\lambda$ along a curve has appeared before in the study of the Square and Rectangular Peg Problems. 
Tao~\cite{TAO2017} developed an integration approach to the Square Peg Problem; Greene--Lobb~\cite{GL24floerhomologysquarepegs} constructed a Floer homology generated by the inscriptions of a rectangle in a real analytic Jordan curve, whose spectral invariants can be described by a Liouville primitive; and Schwartz~\cite{Schwartz} also studied such quantities. 
\medskip

The definition is naturally local along a curve, and it is useful to introduce the arc version. 
By a \emph{Jordan arc} we mean a continuous injective map from a compact interval to $\bR^2$; no condition is imposed on the two endpoints other than being distinct, which follows automatically from injectivity.

\begin{definition}\label{def:arc}
    A Jordan arc $\gamma\colon[a,b]\to\bR^2$ is said to be \emph{Liouville regular} if there exists a sequence of $C^1$-embeddings $\gamma_n\colon[a,b]\to\bR^2$ such that
    \begin{enumerate}[(1)]
    \item $\gamma_n\to\gamma$ uniformly;
    \item for the functions $f_n \colon [a,b] \to \bR$ defined by
    \[
      f_n(t) \coloneqq \int_a^t\gamma_n^*\lambda,
    \]
    there exists a continuous function $f \colon[a,b]\to\bR$ such that $f_n\to f$ uniformly.
    \end{enumerate}
    We again call $f$ a \emph{Liouville primitive} for $\gamma$. 
    A sequence $(\gamma_n,f_n)_n$ satisfying the above conditions is called an \emph{approximating sequence} for $(\gamma,f)$.
\end{definition}

The purpose of this paper is to investigate which curves and arcs are Liouville regular.

Our first criterion is variation-theoretic. 
The arc formulation behaves well under restriction, reparametrization, and finite concatenation, and it is invariant under area-preserving bi-Lipschitz homeomorphisms.
These facts lead to the following local-to-global criterion.

\begin{theorem}\label{thm:main}
    Let $c\colon S^1\to\bR^2$ be a Jordan curve.
    Suppose that $c$ is a cyclic concatenation of finitely many Jordan arcs, each of which is of one of the following three types:
    \begin{enumerate}[(i)]
        \item 
        $\gamma=(x_\gamma,\xi_\gamma)$ with $\xi_\gamma$ of finite $p$-variation and $x_\gamma$ of finite $q$-variation, for some $p,q\geq1$ satisfying
        \[
          \frac1p+\frac1q>1;
        \]
        \item $x_\gamma$ has bounded variation and $\xi_\gamma$ is continuous;
        \item the image of an arc of type (i) or (ii) under an area-preserving bi-Lipschitz homeomorphism of $\bR^2$.
    \end{enumerate}
    Then $c$ is Liouville regular.
\end{theorem}

In particular, a Jordan curve of finite $p$-variation for some $p<2$ is Liouville regular.

Our second criterion is conformal rather than variation-theoretic.
Let $\Omega\subset\bC$ be a bounded Jordan domain and let $\varphi\colon\bD\to\Omega$ be a conformal mapping, where $\bD$ denotes the open unit disk.
We show in \cref{sec:holder} that it is enough for the images of the radial tails $[r,1)e^{i\theta}$ to have lengths tending to zero uniformly in $\theta$.
To state the result more precisely, we introduce a quantity
\begin{equation*}
    \tau_\varphi(r)
    \coloneqq 
    \sup_{\theta\in\bR}
    \int_r^1\left|\varphi'(se^{i\theta})\right|\,ds
    \in[0,\infty].
\end{equation*}

\begin{theorem}[{\cref{thm:radial-tail,cor:holder_bis}}]\label{thm:holder}
    Let $\Omega\subset\bC$ be a bounded Jordan domain.
    If $\lim_{r\nearrow1}\tau_\varphi(r)=0$, then the boundary $\partial\Omega$ is Liouville regular.
    In particular, if $\Omega$ is a H\"older domain, then $\partial\Omega$ is Liouville regular.
    Consequently, the boundary of every bounded Jordan John domain is Liouville regular, and every planar quasicircle is Liouville regular.
\end{theorem}

See \cref{def:horder} for the definition of H\"older domains.
The standard inclusions relevant here are
\[
  \{\text{quasidisks}\}
  \subsetneq
  \{\text{John domains}\}
  \subsetneq
  \{\text{H\"older domains}\};
\]
both inclusions are proper even among simply connected planar domains.

Combining \cref{thm:main,thm:holder} with \cref{thm:peg} gives the following.

\begin{corollary}\label{cor:main-peg}
    Let $R$ be a rectangle in $\bR^2$.
    Every Jordan curve satisfying the hypotheses of \cref{thm:main}, and every curve covered by \cref{thm:holder}, inscribes a rectangle similar to $R$.
\end{corollary}

We emphasize that the class of curves known to inscribe a rectangle of any similarity class is \emph{not} enlarged here: it remains exactly the one obtained in our previous paper~\cite{AI_peg}.

The Koch snowflake is known to be of finite $p$-variation with $p=\frac{\log 4}{\log 3} \approx 1.2619$ and to be a planar quasicircle. 
This implies that it inscribes a rectangle in every prescribed similarity class by \cref{cor:main-peg}.

We give a family of spirals for which Liouville regularity is exactly controlled by convergence of an improper signed-area integral.
We also construct a different non-example that is smooth away from one endpoint and differentiable even at that endpoint. 
Finally, we construct a Liouville regular Jordan arc of positive measure, showing that the approximation condition is not a smallness condition on the image in the sense of Lebesgue measure.
\smallskip

The paper is organized as follows.
In \cref{sec:general} we explore general properties of Liouville regular curves and arcs.
In \cref{sec:variation} we prove the variation criteria, and in \cref{sec:holder} we develop the radial-tail criterion and prove \cref{thm:holder}.
The invariance under area-preserving bi-Lipschitz homeomorphisms is proved in \cref{sec:bilip}.
Finally, \cref{sec:examples,sec:positive} discuss examples, non-examples, and positive-area examples.
\smallskip

While preparing this paper, we learned that Li and Pan~\cite{LiPan} also proved that Jordan curves of finite $p$-variation for some $p<2$ are Liouville regular in the sense of \cref{def:curve}, which, combined with \cref{thm:peg} (that is, \cite[Thm.~1.3]{AI_peg}), implies the rectangular peg theorem for these curves.

\section*{Acknowledgments}

The authors thank Shouhei Honda for asking whether Liouville regularity of a Jordan curve is characterized in terms of a defining function for the domain it bounds. 
This question motivated the study in this paper.
They also thank Shigeki Aida for pointing out references on Young integration.

\section*{Use of AI}

The starting point of this paper was a question we asked ChatGPT and Claude: can the Liouville form be integrated along a fractal arc such as the Koch curve? In reply they pointed us to Young integration, a theory of which the authors then knew very little. This paper would not exist without these AI systems.
More specifically, in writing this paper, we used AI in the following ways.
\begin{itemize}
    \item Part of the proof strategy for \cref{prop:pq-variation,{prop:bv}}, that is, the use of \cref{lem:bg-interpolation} proved in \cite{BoedihardjoGeng}, was proposed by ChatGPT-5.6 Sol.
    
    \item Some of the mathematical arguments in \cref{sec:holder} were generated by ChatGPT-5.6 Sol.
    
    \item Adding the logarithmic factor to the spiral example in \cref{sec:examples} was proposed by ChatGPT-5.6 Sol.
    
    \item We used ChatGPT-5.6 Sol and Claude Opus 5 to assist in writing English explanations and in polishing our manuscript.
\end{itemize}
The authors take full responsibility for the content of this paper.

\section{General properties of Liouville regular curves and arcs}\label{sec:general}

In this section, we explore basic properties of Liouville regular curves and arcs.

\subsection{Elementary lemmas}

Here we prove some lemmas for Liouville regular arcs.

\begin{lemma}\label{lem:diagonal}
    Let $\gamma\colon[a,b]\to\bR^2$ be a Jordan arc.  
    Let $(\widetilde\gamma^{(n)})_n$ be a sequence of Liouville regular Jordan arcs, with Liouville primitives $f^{(n)}$.  
    Assume
    \[
      \widetilde\gamma^{(n)}\to\gamma,
      \quad
      f^{(n)}\to f
    \]
    uniformly, where $f$ is continuous.  
    Then $\gamma$ is Liouville regular with Liouville primitive $f$.
\end{lemma}
\begin{proof}
    For each $n$, choose an approximating sequence for $(\widetilde\gamma^{(n)},f^{(n)})$, denoted by $(\gamma_m^{(n)},f_m^{(n)})_m$.  
    Choose $m(n)$ sufficiently large such that
    \[
        \|\gamma_{m(n)}^{(n)}-\widetilde\gamma^{(n)}\|_\infty<\frac1n,
        \quad
        \|f_{m(n)}^{(n)}-f^{(n)}\|_\infty<\frac1n.
    \]
    Then $(\gamma_{m(n)}^{(n)},f_{m(n)}^{(n)})_n$ is an approximating sequence for $(\gamma,f)$.
\end{proof}

The same argument applies to cyclic parametrizations of Jordan curves, with uniform convergence of primitives on compact subsets of $\bR$.

\begin{lemma}\label{lem:restriction}
    Every subarc of a Liouville regular Jordan arc or Jordan curve is also Liouville regular.
\end{lemma}
\begin{proof}
    It is enough to restrict an approximating sequence to the corresponding compact interval and subtract the value of its primitive at the left endpoint.
\end{proof}

\begin{lemma}\label{lem:reparam}
    Liouville regularity is invariant under reparametrization by a homeomorphism of compact intervals.
    The circle version is invariant under homeomorphisms of $S^1$.
\end{lemma}
\begin{proof}
    Let $h\colon J\to I$ be a homeomorphism of compact intervals. 
    Approximate $h$ uniformly by $C^1$-diffeomorphisms $h_n\colon J\to I$ having the same orientation as $h$. 
    If $(\gamma_n,f_n)_n$ approximates $(\gamma,f)$, then $\gamma_n\circ h_n$ is a $C^1$-embedding and its normalized primitive is
    \[
      f_n\circ h_n-f_n(h_n(\min J)).
    \]
    Uniform convergence follows from the uniform continuity of $f$ and the uniform convergence of $h_n$.  
    The circle case is obtained by periodic lifts of the homeomorphisms.
\end{proof}

\begin{proposition}\label{prop:C2_area_pres}
    Let $\gamma \colon [a,b] \to \bR^2$ be a Liouville regular arc and $\phi \colon \bR^2 \to \bR^2$ be an orientation-preserving area-preserving $C^2$-diffeomorphism.
    Then the curve $\phi \circ \gamma$ is also Liouville regular.
\end{proposition}
\begin{proof}
    Since $\phi^* d\lambda = d\lambda$, we have $d(\phi^*\lambda - \lambda)=0$ on $\bR^2$.
    Hence, there exists a $C^2$-function $h$ on $\bR^2$ such that $\phi^*\lambda - \lambda=dh$.

    Let $f$ be a Liouville primitive for $\gamma$ and choose an approximating sequence $(\gamma_n,f_n)_n$ for $(\gamma,f)$.
    Then $\phi \circ \gamma_n$ is of class $C^1$ and $\phi \circ \gamma_n \to \phi \circ \gamma$ uniformly.
    We have 
    \begin{align*}
        \int_a^t (\phi \circ \gamma_n)^* \lambda 
        & = 
        \int_a^t \gamma_n^* \phi^* \lambda \\
        & = 
        \int_a^t \gamma_n^* \lambda + \int_a^t \gamma_n^* dh
        = 
        f_n(t) + h(\gamma_n(t)) - h(\gamma_n(a)).
    \end{align*}
    As $n \to \infty$, this converges to the continuous function $t \mapsto f(t) + h(\gamma(t)) - h(\gamma(a))$ uniformly.
\end{proof}

We will generalize the above proposition to the bi-Lipschitz case in \cref{sec:bilip}.

\subsection{Rounding polygonal embeddings}

We shall repeatedly use embedded polygonal approximations.  The following elementary observation allows us to pass from them to the $C^1$-approximations required in \cref{def:arc,def:curve}.

\begin{lemma}\label{lem:polygonal-rounding}
    Let $P\colon[a,b]\to\bR^2$ be an embedded polygonal arc, parametrized linearly on each edge.  
    For every $\varepsilon>0$ there exists a $C^1$-embedded arc $\widetilde P\colon[a,b]\to\bR^2$ with the same endpoints such that
    \[
      \|\widetilde P-P\|_\infty<\varepsilon
    \]
    and, after normalizing both primitives to vanish at $a$,
    \[
      \sup_{t\in[a,b]}
      \left|
        \int_a^t\widetilde P^*\lambda-
        \int_a^t P^*\lambda
      \right|<\varepsilon.
    \]
    The analogous statement holds for embedded polygonal Jordan curves, with the primitives compared on one period.
\end{lemma}
\begin{proof}
    Choose pairwise disjoint disks around the finitely many vertices, so small that each disk meets the polygon only in the two incident edge pieces. 
    Inside each disk replace the corner by a $C^1$-embedded rounded corner agreeing with the polygon near the boundary of the disk. 
    The replacement can be chosen arbitrarily $C^0$-small and with length bounded by a fixed multiple of the disk radius. 
    On a compact neighborhood of the polygon the norm of $\lambda$ is bounded. 
    Hence, the change of the integral of $\lambda$ produced by each rounding tends to zero with the radius of its disk. 
    Taking the disks sufficiently small, and distributing the error among the finitely many vertices, gives the asserted uniform estimate for the primitives. 
    The cyclic case is similar.
\end{proof}

\subsection{Concatenation and closing}

Suppose that $\gamma^1\colon[0,b_1]\to\bR^2$ and
$\gamma^2\colon[0,b_2]\to\bR^2$ satisfy
$\gamma^1(b_1)=\gamma^2(0)$.  We write $\gamma^2*\gamma^1$ for their usual concatenation on $[0,b_1+b_2]$.

\begin{proposition}\label{prop:concat}
    For $i=1,2$, let $\gamma^i\colon[0,b_i]\to\bR^2$ be a Liouville regular Jordan arc.
    Assume $\gamma^1(b_1)=\gamma^2(0)$ and that the concatenation $\gamma^2*\gamma^1$ is injective.  
    Then $\gamma^2*\gamma^1$ is Liouville regular.
    If $f^i$ is a Liouville primitive for $\gamma^i$ normalized by $f^i(0)=0$, then a Liouville primitive for $\gamma^2*\gamma^1$ is given by
    \[
     f(t)=
     \begin{cases}
       f^1(t),& 0\leq t\leq b_1,\\
       f^1(b_1)+f^2(t-b_1),&b_1\leq t\leq b_1+b_2.
     \end{cases}
    \]
\end{proposition}
\begin{proof}
    Put $P=\gamma^1(b_1)=\gamma^2(0)$.  
    Let $(\gamma_n^i,f_n^i)_n$ be approximating sequences for $(\gamma^i,f^i)$.

    Let $n \ge 1$ and consider the open disk $B(P;1/n)$ centered at $P$ of radius $1/n$.
    There exists $\varepsilon_n>0$ such that the distance between $\Image(\gamma^1) \cap \partial B(P;1/n)$ and $\Image(\gamma^2) \cap \partial B(P;1/n)$ is at least $3\varepsilon_n$.
    Take a sufficiently large $m(n)$ such that $\gamma^i_{m(n)}$ is $\varepsilon_n$-close to $\gamma^i$ in the $C^0$-sense for $i=1,2$. 
    If necessary, we can apply a $C^1$-small perturbation supported in $B(P;1/n+\varepsilon_n)$ and assume that $\gamma^i_{m(n)}$ is transverse to $\partial B(P;1/n)$. 
    Let 
    \begin{align*}
        s^1_n & \coloneqq \min \{ t \in [0,b_1] \mid |\gamma^1_{m(n)}(t) -P | \le 1/n \} \ \text{and} \\
        s^2_n & \coloneqq \max \{ t \in [0,b_2] \mid |\gamma^2_{m(n)}(t)-P| \le 1/n \}
    \end{align*}
    be the first time at which $\gamma^1_{m(n)}$ reaches the closed disk $\overline{B(P;1/n)}$ and the last time at which $\gamma^2_{m(n)}$ leaves it, respectively.
    As $n \to \infty$, we have $s^1_n \to b_1$ and $s^2_n \to 0$.
    We retain the pieces $\gamma^1_{m(n)}|_{[0,s^1_n]}$ and $\gamma^2_{m(n)}|_{[s^2_n,b_2]}$ and discard the rest. 
    The retained pieces meet $\overline{B(P;1/n)}$ only at their new endpoints.  
    Extend these pieces by sufficiently small disjoint segments in $\overline{B(P;1/n)}$ and join the new endpoints by another segment to get a piecewise $C^1$-arc whose part in $\overline{B(P;1/n)}$ is an embedded polygonal arc. 
    By applying \cref{lem:polygonal-rounding} to the polygonal part and reparametrizing the resulting arc, we obtain a $C^1$-embedding $\gamma_n \colon [0,b_1+b_2] \to \bR^2$.
    
    Then the resulting sequence $(\gamma_n)_n$ converges uniformly to $\gamma^2*\gamma^1$.
    Since $f_n^i\to f^i$ uniformly and the $f^i$ are continuous, the variations of the primitives on the discarded pieces tend to zero.
    The connecting subarc in $\gamma_n$ can be chosen with length $O(1/n)$; since $\lambda$ is bounded near $P$, its integral of $\lambda$ is $O(1/n)$.  
    Consequently the primitives of the glued arcs converge uniformly to the function $f$ in the statement.
\end{proof}

\begin{proposition}\label{prop:closing}
    Let $c\colon S^1\to\bR^2$ be a Jordan curve.  
    Suppose there are finitely many points
    \[
      \theta_0<\theta_1<\cdots<\theta_N=\theta_0+2\pi
    \]
    such that every arc
    \[
      c\circ e\big|_{[\theta_{i-1},\theta_i]}
    \]
    is Liouville regular.
    Then $c$ is Liouville regular.
\end{proposition}
\begin{proof}
    By \cref{prop:concat}, we can reduce the problem to the case $N=2$.
    Then it is enough to perform the local surgery from the proof of \cref{prop:concat} simultaneously at the two connecting points.
\end{proof}

\begin{corollary}\label{cor:closing-peg}
    Let $c\colon S^1\to\bR^2$ be a Jordan curve which is a finite union of Liouville regular Jordan arcs.
    Then $c$ inscribes a rectangle in every prescribed similarity class.
\end{corollary}
\begin{proof}
    Combine \cref{prop:closing} with \cref{thm:peg}.
\end{proof}

\section{Arcs of finite variation}\label{sec:variation}

In this section, we explore the relation between arcs with finite variation and Liouville regularity.

\subsection{Variation and Young integration}

We briefly recall the $p$-variation and the Young integral.
See \cite{FrizVictoir} for details.

Let $(\bV,\|\cdot\|)$ be a normed vector space and $p\geq1$.  
For $u\colon[a,b]\to \bV$, define
\[
  \|u\|_{\pvar p}
  \coloneqq 
  \sup_{\Delta}
  \left(
    \sum_{i=0}^{N-1} \|u(t_i)-u(t_{i+1})\|^p
  \right)^{1/p},
\]
where $\Delta=\{a=t_0<\cdots<t_N=b\}$ ranges over finite partitions. 
We denote by $V^p([a,b],\bV)$ the space of maps of finite $p$-variation, that is, 
\[
    V^p([a,b],\bV) \coloneqq \{u \colon [a,b] \to \bV \mid \|u\|_{\pvar p} <+\infty \}.
\]
The case $p=1$ corresponds to bounded variation.

\begin{lemma}[{\cite[Prop.~5.5]{FrizVictoir}}]\label{lem:pvar-interpolation}
    Let $p \ge 1$ and $u \in V^p([a,b],\bV)$.
    For $1 \le p < p' <\infty$, one has 
    \begin{equation*}
        \|u\|_{\pvar{p'}}^{p'} \le \osc(u)^{p'-p} \|u\|_{\pvar p}^{p},
    \end{equation*}
    where $\osc(u) \coloneqq \sup_{t,s \in [a,b]} \|u(t)-u(s)\|$.
    In particular, for $v \in V^p([a,b],\bV)$, one has  
    \[
        \|v\|_{\pvar{p'}}^{p'} \le (2\|v\|_\infty)^{p'-p} \|v\|_{\pvar p}^{p}.
    \]
\end{lemma}

Let $u,v\colon[a,b]\to\bR$ be continuous functions. 
If $v$ is of bounded variation, one can define the Riemann--Stieltjes integral
\begin{equation}\label{eq:integral}
    \int_a^b u \, dv 
    \coloneqq 
    \lim_{\mesh{\Delta}} \sum_{i=0}^{N-1} u(t_i) (v(t_{i+1})-v(t_i)),
\end{equation}
where $\Delta = \{a=t_0<t_1<\dots<t_N=b \}$ and $\mesh{\Delta}=\max_{i} (t_{i+1}-t_i)$.
If $u\in V^p([a,b],\bR)$, $v\in V^q([a,b],\bR)$, and $p^{-1}+q^{-1}>1$, the same limit in \eqref{eq:integral} defines the Young integral \cite{Young}. 
The Young--Lo\`eve estimate asserts that the Young integral is continuous with respect to the corresponding variation norms, that is, 
\[
    \left| \int_a^b u \, dv -u(a) (v(b)-v(a)) \right| \le 
    \frac{1}{1-2^{1-\theta}} \|u\|_{\pvar p} \|v\|_{\pvar q}
\]
with $\theta = p^{-1}+q^{-1}>1$. 

\subsection{Embedded polygonal interpolation}

We write $(x,\xi)$ for the coordinates on $\bR^2$.
Hence $x \colon \bR^2 \to \bR$ denotes the first projection and $\xi \colon \bR^2 \to \bR$ denotes the second projection. 
For an arc $\gamma\colon[a,b]\to\bR^2$, write
\[
    x_\gamma \coloneqq x\circ\gamma,
    \quad
    \xi_\gamma \coloneqq \xi\circ\gamma.
\]

For a continuous real-valued function $u$ and a partition
\[
    \Delta=\{a=t_0<\cdots<t_N=b\},
\]
write $u^\Delta$ for the piecewise affine interpolation satisfying $u^\Delta(t_i)=u(t_i)$.  
For an arc $\gamma$, set
\[
    \gamma^\Delta \coloneqq (x_\gamma^\Delta,\xi_\gamma^\Delta).
\]

The following is the Euclidean case of the approximation theorem of Boedihardjo--Geng~\cite{BoedihardjoGeng}.

\begin{lemma}[{Simple polygonal interpolation \cite[Thm.~2.1, Thm.~2.2, and Rem.~2.3]{BoedihardjoGeng}}]\label{lem:bg-interpolation}
    Let $\gamma\colon[a,b]\to\bR^2$ be a Jordan arc, and let $a<a_1<\cdots<a_k<b$ be finitely many prescribed parameter values.  For every $\varepsilon>0$ there exists a finite partition $\Delta$ of $[a,b]$ such that
    \begin{enumerate}[(1)]
    \item $a_1,\ldots,a_k\in\Delta$;
    \item $\mesh\Delta<\varepsilon$;
    \item $\gamma^\Delta$ is an embedded polygonal arc.
    \end{enumerate}
    For a Jordan curve there are arbitrarily fine partitions for which the periodic polygonal interpolation is a Jordan polygon.
\end{lemma}

\begin{lemma}\label{lem:polygonal-var}
Let $p \ge 1$ and $u\colon[a,b]\to\bR$ be continuous and of finite $p$-variation.  Then for every finite partition $\Delta$,
\[
  \|u^\Delta\|_{\pvar p}
  \leq
  \|u\|_{\pvar p}.
\]
If $\mesh\Delta\to0$, then
\[
  \|u^\Delta-u\|_\infty\to0,
\]
and, for every $p'>p$,
\[
  \|u^\Delta-u\|_{\pvar{p'}}\to0.
\]
\end{lemma}
\begin{proof}
For a real-valued piecewise affine function, the supremum defining its $p$-variation can be taken over partitions consisting of vertices.  Indeed, if a partition point lies in the interior of an affine edge, its value lies between the endpoint values; convexity of
\[
 z\mapsto |A-z|^p+|z-B|^p
\]
shows that moving it to one of the two endpoints cannot decrease the relevant variation sum.  Since the vertex values of $u^\Delta$ are values of $u$, the first inequality follows.

Denoting the modulus of continuity of $u$ by $\omega_u$, we have
\[
  \|u^\Delta-u\|_\infty\leq\omega_u(\mesh\Delta),
\]
which implies that $u^\Delta\to u$ uniformly. 
Moreover,
\[
  \|u^\Delta-u\|_{\pvar p}\leq2\|u\|_{\pvar p}.
\]
By applying \cref{lem:pvar-interpolation} to $u^\Delta-u$, for $p'>p$, we obtain
\begin{align*}
    \|u^\Delta-u\|_{\pvar{p'}}
    & \le 
    (2\|u^\Delta-u\|_\infty)^{1-p/p'} \|u^\Delta-u\|_{\pvar{p}}^{p/p'} \\
    & \le 
    (2\|u\|_{\pvar{p}})^{p/p'} (2\|u^\Delta-u\|_\infty)^{1-p/p'},
\end{align*}
which tends to zero. 
\end{proof}

The polygonal interpolations in \cref{lem:bg-interpolation} are not $C^1$ at their vertices.  
This causes no problem: after proving convergence of their primitives, we apply \cref{lem:polygonal-rounding} with an error tending to zero.  
Equivalently, one may regard this as an application of the diagonal lemma (\cref{lem:diagonal}).

\subsection{Complementary variation exponents}

\begin{proposition}\label{prop:pq-variation}
Let $\gamma\colon[a,b]\to\bR^2$ be a Jordan arc.  Assume
\[
  \xi_\gamma\in V^p([a,b],\bR),
  \quad
  x_\gamma\in V^q([a,b],\bR),
  \quad
  \frac1p+\frac1q>1.
\]
Then $\gamma$ is Liouville regular and a Liouville primitive for $\gamma$ is given by the Young integral
\[
  f(t)=\int_a^t\xi_\gamma\,dx_\gamma.
\]
\end{proposition}
\begin{proof}
Choose partitions $\Delta_n$ as in \cref{lem:bg-interpolation}, with $\mesh{\Delta_n}\to0$, and write
\[
  \gamma^{\Delta_n}=(x_n,\xi_n).
\]
Each $\gamma^{\Delta_n}$ is an embedded polygonal arc, and the sequence $(\gamma^{\Delta_n})_n$ converges uniformly to $\gamma$.

Choose $p'>p$ and $q'>q$ sufficiently close to $p$ and $q$ such that
\[
  \frac1{p'}+\frac1{q'}>1.
\]
By \cref{lem:polygonal-var},
\[
  \|\xi_n-\xi_\gamma\|_{\pvar{p'}}\to0,
  \quad
  \|x_n-x_\gamma\|_{\pvar{q'}}\to0.
\]
The Young--Lo\`eve estimate therefore implies
\[
  \sup_{t\in[a,b]}
  \left|
    \int_a^t\xi_n\,dx_n-
    \int_a^t\xi_\gamma\,dx_\gamma
  \right|
  \to0.
\]
Finally round $\gamma^{\Delta_n}$ using \cref{lem:polygonal-rounding}, choosing the $C^0$ and primitive errors smaller than $1/n$.  
The rounded arcs form the required $C^1$-approximating sequence.
\end{proof}

\begin{corollary}\label{cor:pvar}
    If a Jordan arc is of finite $p$-variation for some $p<2$, then it is Liouville regular.
    In particular, every rectifiable Jordan arc is Liouville regular.
\end{corollary}
\begin{proof}
Both coordinate functions have finite $p$-variation, so \cref{prop:pq-variation} applies with the same exponent in the two coordinates.  The rectifiable case is $p=1$.
\end{proof}

In particular, this gives another proof of~\cite[Prop.~5.8]{AI_peg}.

\begin{remark}[Relation with geometric $2$-rough paths]\label{rem:rough-path}
    If a Jordan arc $\gamma$ can be enhanced to a geometric $2$-rough path, then it is Liouville regular.
    This follows from Yang~\cite[Thm.~23]{YangArea} and the simple polygonal interpolation of \cref{lem:bg-interpolation}.     
\end{remark}

\subsection{One coordinate of bounded variation}

If one coordinate is of bounded variation, no variation bound is needed
for the other coordinate.

\begin{proposition}\label{prop:bv}
    Let $\gamma\colon[a,b]\to\bR^2$ be a Jordan arc.  
    Assume that $x_\gamma$ has bounded variation.  
    Then $\gamma$ is Liouville regular and a Liouville primitive for $\gamma$ is given by the Riemann--Stieltjes integral
    \[
        f(t)=\int_a^t\xi_\gamma\,dx_\gamma.
    \]
\end{proposition}
\begin{proof}
    Choose simple polygonal interpolations
    \[
      \gamma^{\Delta_n}=(x_n,\xi_n)
    \]
    with $\mesh{\Delta_n}\to0$.  
    By \cref{lem:polygonal-var}, $\|x_n\|_{\pvar 1} \le \|x_\gamma\|_{\pvar 1}$ and 
    \[
      x_n\to x_\gamma,
      \quad
      \xi_n\to\xi_\gamma
    \]
    uniformly.
    
    We claim that the Riemann--Stieltjes integrals converge uniformly.
    Since 
    \[
     \sup_t
     \left|
      \int_a^t(\xi_n-\xi_\gamma)\,dx_n
     \right|
     \leq
     \|\xi_n-\xi_\gamma\|_\infty\|x_n\|_{\pvar 1}
     \to0,
    \]
    it remains to show
    \[
     \sup_t
     \left|
      \int_a^t\xi_\gamma\,d(x_n-x_\gamma)
     \right|\to0.
    \]
    Fix $\varepsilon>0$ and approximate the continuous function $\xi_\gamma$ uniformly by a step function $s$ whose discontinuity points form a fixed finite partition.  Since
    \[
     \|x_n-x_\gamma\|_{\pvar 1}
     \leq 2\|x_\gamma\|_{\pvar 1},
    \]
    the error caused by replacing $\xi_\gamma$ with $s$ is bounded uniformly in $n$ by a constant times $\|\xi_\gamma-s\|_\infty$.  
    For the fixed step function $s$, the integral against $d(x_n-x_\gamma)$ is a finite linear combination of increments of $x_n-x_\gamma$ at the partition points, together with at most one final incomplete increment.      
    Hence, $\sup_t | \int_a^t s \, d(x_n-x_\gamma) | \to 0$ as $n \to \infty$, by the uniform convergence $x_n\to x_\gamma$.  
    Letting $\varepsilon\to0$ proves the claim.
    
    Thus
    \[
     \int_a^t\xi_n\,dx_n
     \to
     \int_a^t\xi_\gamma\,dx_\gamma
    \]
    uniformly in $t$.  
    Apply \cref{lem:polygonal-rounding} to replace each polygonal embedding by a $C^1$-embedding with vanishing additional error.
    This completes the proof.
\end{proof}

\section{A radial-tail criterion and H\"older domains}\label{sec:holder}

We now give a sufficient condition adapted to conformal approximations.
We start with a more general criterion for the boundary of a Jordan domain.

Let $\Omega\subset\bC\simeq\bR^2_{x,\xi}$ be a bounded Jordan domain and let $\bD$ denote the open unit disk in $\bC$. 
Assume that we are given a homeomorphism between their closures $\overline{\bD}\to\overline\Omega$ whose restriction to $\bD$ is a $C^1$-diffeomorphism $\varphi \colon \bD \to \Omega$.
For $0<r<1$, define a $C^1$ Jordan curve $c_r\colon S^1\to\Omega$ by
\[
    c_r(\theta) \coloneqq \varphi(re^{i\theta}),
\]
and let $c(\theta)=\varphi(e^{i\theta})$ parametrize $\partial\Omega$. 

The quantity relevant to the convergence of the primitives is the maximal length of a radial tail:
\begin{equation*}\label{eq:radial-tail}
  \tau_\varphi(r)
  \coloneqq 
  \sup_{\theta\in\bR}
  \int_r^1\left|\varphi'(se^{i\theta})\right|\,ds
  \in[0,\infty].
\end{equation*}
Thus $\tau_\varphi(r)$ is the supremum, over all directions, of the lengths of the curves
$\varphi([r,1)e^{i\theta})$.
We prove the first part of \cref{thm:holder}.

\begin{theorem}[Radial-tail criterion]\label{thm:radial-tail}
    Suppose that
    \begin{equation}\label{eq:tail-condition}
      \lim_{r\nearrow1}\tau_\varphi(r)=0.
    \end{equation}
    Then the boundary $\partial\Omega$ is Liouville regular.
    More precisely, defining 
    \[
        f_r(t) \coloneqq \int_0^t(c_r\circ e)^*\lambda,
        \quad t\in\bR
    \]
    and 
    \[
        M \coloneqq \sup_{z\in\overline\Omega}|\xi(z)|,
        \quad
        E(r) \coloneqq \Area\bigl(\Omega\setminus\varphi(r\bD)\bigr),
    \]
    one has
    \begin{equation}\label{eq:primitive-tail-estimate}
      \sup_{0\leq t\leq2\pi}|f_s(t)-f_r(t)|
      \leq E(r)+2M\tau_\varphi(r)
    \end{equation}
    for $0<r<s<1$.
    In particular, there exists a continuous function $f$ on $\bR$ such that $f_r \to f$ uniformly on every compact subset of $\bR$ as $r\nearrow1$.
\end{theorem}
\begin{proof}
    The curves $c_r$ are $C^1$-embeddings and converge uniformly to $c$ as $r\nearrow1$.
    Fix $0<r<s<1$ and $t\in[0,2\pi]$.
    The image under $\varphi$ of
    \[
      Q_{r,s,t}
      =\{\rho e^{i\theta} \mid r\leq\rho\leq s,\ 0\leq\theta\leq t\}
    \]
    is bounded by the two arcs $c_r|_{[0,t]}$ and $c_s|_{[0,t]}$ and by the two radial sides
    \[
      \eta_{\theta,r,s}(u)=\varphi(ue^{i\theta}),
      \quad r\leq u\leq s,
      \quad \theta\in\{0,t\}.
    \]
    Since $d\lambda=d\xi\wedge dx$, Green's theorem gives
    \begin{equation*}
      |f_s(t)-f_r(t)|
      \leq
      \Area\bigl(\varphi(Q_{r,s,t})\bigr)
      +\left|\int_{\eta_{0,r,s}}\lambda\right|
      +\left|\int_{\eta_{t,r,s}}\lambda\right|.
    \end{equation*}
    The area term is at most $E(r)$.
    Moreover, $|\lambda(v)|\leq M|v|$ on $\overline\Omega$, and hence each radial term is bounded by
    \[
      M\int_r^s|\varphi'(ue^{i\theta})|\,du
      \leq M\tau_\varphi(r).
    \]
    This proves \eqref{eq:primitive-tail-estimate}.
    
    The sets $\varphi(r\bD)$ increase to $\Omega$ as $r\nearrow1$.
    Since $\Omega$ is bounded, continuity from below of planar Lebesgue measure gives $E(r)\to0$.
    Together with \eqref{eq:tail-condition}, the estimate shows that $(f_r)_r$ is uniformly Cauchy on $[0,2\pi]$.
    Finally,
    \[
      f_r(t+2\pi)=f_r(t)+f_r(2\pi),
    \]
    so convergence on one period implies locally uniform convergence on $\bR$.
    The limit is continuous.  Choosing any sequence $r_n\nearrow1$, the curves $c_{r_n}$ give the approximating sequence required in \cref{def:curve}.
\end{proof}

We next apply the criterion to H\"older domains. 
Let $\Omega$ be a bounded simply connected domain in $\bC$.
By the Riemann mapping theorem, there exists a biholomorphism $\varphi \colon \bD \to \Omega$.
If $\Omega$ is a Jordan domain, $\varphi$ extends to a homeomorphism $\overline{\bD}\to\overline\Omega$ by Carath\'eodory's theorem.
Hence, we can apply the radial-tail criterion in this case.

\begin{definition}\label{def:horder}
    A bounded simply connected domain $\Omega \subset \bC$ is said to be a \emph{H\"older domain} if some (equivalently, every) biholomorphism $\varphi\colon\bD\to\Omega$ is H\"older continuous, that is, if there exist $0<\alpha\leq1$ and $L>0$ such that
    \begin{equation}\label{eq:holder-riemann-map}
        |\varphi(z)-\varphi(w)|\leq L|z-w|^\alpha,
        \quad z,w\in\bD.
    \end{equation}
\end{definition}

This definition is equivalent to the logarithmic growth condition for the quasihyperbolic metric; see, for example,~\cite[Sect.~4.6]{Pommerenke92}.

The following estimate is the easy direction of the classical
Hardy--Littlewood characterization of analytic H\"older functions;
see, for example, \cite[p.~411]{Goluzin}.
We include the short proof for completeness.

\begin{lemma}\label{lem:holder-derivative}
Let $\varphi \colon \bD \to \Omega$ be a biholomorphism.
If \eqref{eq:holder-riemann-map} holds, then
\begin{equation}\label{eq:holder-derivative}
  |\varphi'(z)|
  \leq 2^{1-\alpha}L(1-|z|)^{\alpha-1},
  \quad z\in\bD.
\end{equation}
Consequently,
\begin{equation}\label{eq:holder-tail-bound}
  \tau_\varphi(r)
  \leq \frac{2^{1-\alpha}L}{\alpha}(1-r)^\alpha.
\end{equation}
\end{lemma}

\begin{proof}
Fix $z\in\bD$ and set $\delta \coloneqq (1-|z|)/2$.
Applying Cauchy's formula on $|w-z|=\delta$ and subtracting the constant $\varphi(z)$ gives
\[
  \varphi'(z)
  =\frac{1}{2\pi i}
   \int_{|w-z|=\delta}
   \frac{\varphi(w)-\varphi(z)}{(w-z)^2}\,dw.
\]
The H\"older estimate therefore yields
\[
  |\varphi'(z)|\leq L\delta^{\alpha-1}
  =2^{1-\alpha}L(1-|z|)^{\alpha-1},
\]
which is \eqref{eq:holder-derivative}.
Integrating this bound along each radius from $r$ to $1$ proves \eqref{eq:holder-tail-bound}.
\end{proof}

The latter part of \cref{thm:holder} is a consequence of \cref{thm:radial-tail,lem:holder-derivative}.

\begin{corollary}\label{cor:holder_bis}
    Let $\Omega$ be a bounded Jordan domain.
    If $\Omega$ is a H\"older domain, then the boundary $\partial\Omega$ is Liouville regular.
    Consequently, the boundary of every bounded Jordan John domain is Liouville regular, and every planar quasicircle is Liouville regular.
\end{corollary}
\begin{proof}
    If $\Omega$ is a H\"older domain, \cref{lem:holder-derivative} implies the radial-tail condition \eqref{eq:tail-condition}; hence \cref{thm:radial-tail} shows that $\partial\Omega$ is Liouville regular.
    
    Every bounded simply connected John domain is a H\"older domain~\cite[p.~103]{Pommerenke92}, while every quasidisk is a John domain~\cite[Thm.~5.9]{Pommerenke92}.
    The bounded complementary component of a planar quasicircle is a quasidisk.
    Hence, the remaining assertions follow.
\end{proof}

By Rohde~\cite{Rohde}, a planar quasicircle is of finite $p$-variation for some $p<2$.
Thus \cref{cor:pvar} also implies that every planar quasicircle admits an inscribed rectangle in every prescribed similarity class.

\begin{remark}
    Let $\Omega$ be a Jordan domain and $\overline{\bD} \to \overline{\Omega}$ be a homeomorphism  whose restriction to $\bD$ is a $C^1$-diffeomorphism $\psi \colon \bD \to \Omega$.
    If $\tau_{\psi}(r) \to 0$ as $r\nearrow1$, then for any biholomorphism $\varphi \colon \bD \to \Omega$, one has $\tau_{\varphi}(r) \to 0$ as $r\nearrow1$.
    This follows from the Gehring--Hayman inequality~\cite[Thm.~4.20]{Pommerenke92}, which states that the Euclidean length of the hyperbolic geodesic is bounded above by a constant multiple of the length of any arc connecting the same endpoints.
\end{remark}

\section{Area-preserving bi-Lipschitz homeomorphisms}\label{sec:bilip}

We next show that the property of being Liouville regular is invariant under a bi-Lipschitz homeomorphism preserving Lebesgue measure on $\bR^2$. 

\begin{lemma}\label{lem:cocycle}
    Let $\phi\colon\bR^2\to\bR^2$ be an orientation-preserving area-preserving bi-Lipschitz homeomorphism.
    There exists a locally Lipschitz function $S_\phi\colon\bR^2\to\bR$, unique up to an additive constant, such that for every rectifiable arc $\gamma\colon[s,t]\to\bR^2$,
    \[
      \int_{\phi\circ\gamma}\lambda-
      \int_\gamma\lambda
      =
      S_\phi(\gamma(t))-S_\phi(\gamma(s)).
    \]
\end{lemma}
\begin{proof}
    For a rectifiable closed curve $\Gamma$, Green's theorem for rectifiable curves (see, for example, \cite{CVGreen}) gives
    \[
      \int_\Gamma\lambda
      =-
      \int_{\bR^2}\Ind_\Gamma(z)\,dx \wedge d\xi,
    \]
    where $\Ind_\Gamma(z)$ denotes the winding number of $\Gamma$ around $z$, with the sign determined by $d\lambda=d\xi\wedge dx$. 
    Since $\phi$ is orientation preserving,
    \[
      \Ind_{\phi\circ\Gamma}(\phi(z))=\Ind_\Gamma(z).
    \]
    Since $\phi$ preserves Lebesgue measure,
    \[
      \int_{\phi\circ\Gamma}\lambda
      =
      \int_\Gamma\lambda.
    \]
    Consequently
    \[
      \cA_\phi(\gamma)
      \coloneqq 
      \int_{\phi\circ\gamma}\lambda-
      \int_\gamma\lambda
    \]
    depends only on the endpoints.  
    Fix $z_0$ and define $S_\phi(z)$ by evaluating $\cA_\phi$ on any rectifiable arc from $z_0$ to $z$.
    
    To prove local Lipschitz continuity, take $z,w$ in a fixed compact set and evaluate $S_\phi(z)-S_\phi(w)$ on the line segment from $w$ to $z$.  
    The segment and its image have total length at most $(1+\Lip(\phi))|z-w|$, while the $\xi$-coordinates of both arcs remain bounded on a fixed compact set.  
    Thus
    \[
      |S_\phi(z)-S_\phi(w)|\leq C|z-w|,
    \]
    which completes the proof.
\end{proof}

\begin{theorem}\label{thm:bilip}
    Let $\phi\colon\bR^2\to\bR^2$ be an area-preserving bi-Lipschitz homeomorphism.
    A Jordan arc $\gamma \colon [a,b] \to \bR^2$ is Liouville regular if and only if $\phi \circ \gamma$ is Liouville regular.
    The same holds for Jordan curves.
\end{theorem}
\begin{proof}
    It is enough to prove one direction, since the inverse $\phi^{-1}$ has the same properties. 
    We may assume that $\phi$ is orientation-preserving by composing it with the map $(x,\xi) \mapsto (x,-\xi)$.
    
    Let $f$ be a Liouville primitive for a Liouville regular arc $\gamma$ and $(\gamma_n,f_n)_n$ be an approximating sequence for $(\gamma,f)$; the curve case is identical on each period.  
    Set $d_n \coloneqq \phi\circ\gamma_n$.
    Then each $d_n$ is a Lipschitz embedded arc.
    By \cref{lem:cocycle}, after fixing the additive constant in $S_\phi$, a natural Liouville primitive along $d_n$ is given by 
    \[
      g_n \coloneqq f_n+S_\phi\circ\gamma_n-S_\phi(\gamma_n(a)).
    \]
    Since $S_\phi$ is uniformly continuous on a compact neighborhood of the images of the $\gamma_n$,
    \[
      g_n\to
      f+S_\phi\circ\gamma-S_\phi(\gamma(a))
    \]
    uniformly.
    
    The first coordinate of each Lipschitz arc $d_n$ has bounded variation, so $g_n$ is exactly the function given by the Riemann--Stieltjes integral along $d_n$ with normalization.
    Hence, \cref{prop:bv} shows that $d_n$ is Liouville regular with Liouville primitive $g_n$.
    Apply \cref{lem:diagonal} to the sequence $(d_n,g_n)_n$ to conclude that $\phi\circ\gamma$ is Liouville regular.
\end{proof}

\begin{proof}[Proof of \cref{thm:main}]
Arcs of type \textup{(i)} are Liouville regular by \cref{prop:pq-variation}; arcs of type \textup{(ii)} are Liouville regular by \cref{prop:bv}; and arcs of type \textup{(iii)} are Liouville regular by \cref{thm:bilip}.  The finite cyclic concatenation is then Liouville regular by \cref{prop:closing}.
\end{proof}

\section{Examples and non-examples}\label{sec:examples}

In this section, we give examples and non-examples of Liouville regular arcs.

\subsection{Spirals}

For $\alpha>0$ and $\beta\in\bR$, set
\[
  \rho_{\alpha,\beta}(\theta)
  \coloneqq 
  \theta^{-\alpha}(\log\theta)^{-\beta},
  \quad \theta\geq \mathrm e,
\]
and
\[
  \gamma_{\alpha,\beta}(\theta)
  \coloneqq 
  \bigl(
    \rho_{\alpha,\beta}(\theta)\cos\theta,
    \rho_{\alpha,\beta}(\theta)\sin\theta
  \bigr).
\]
After deleting a harmless initial segment, the radius is strictly decreasing.  We compactify the parameter interval by adding $\gamma_{\alpha,\beta}(\infty)=0$, obtaining a Jordan arc $\gamma_{\alpha,\beta} \colon [a_{\alpha,\beta},\infty] \to \bR^2$.

\begin{lemma}\label{lem:spiral-action}
    The improper integral
    \[
      \int_{\mathrm e}^{\infty}\xi(\theta)\,dx(\theta)
    \]
    converges if and only if
    \[
      2\alpha>1,
      \quad\text{or}\quad
      2\alpha=1\ \text{ and }\ 2\beta>1.
    \]
    In all other cases it diverges to $-\infty$.
\end{lemma}
\begin{proof}
Writing $\rho=\rho_{\alpha,\beta}$, we have
\[
\begin{split}
  \xi(\theta) x'(\theta)
  &=\rho\rho'\sin\theta\cos\theta-\rho^2\sin^2\theta\\
  &=\frac12\rho\rho'\sin(2\theta)
    +\frac12\rho^2\cos(2\theta)-\frac12\rho^2\\
  &=\frac{d}{d \theta }\left(\frac14 \rho^2 \sin (2\theta )\right)-\frac12\rho^2.
\end{split}
\]
The improper integral of the first term converges since 
\[\frac14 \rho(\theta)^2 \sin (2\theta )\to 0 \quad (\theta\rightarrow \infty).
\] 
Thus convergence is equivalent to
\[
  \int_{\mathrm e}^{\infty}\theta^{-2\alpha}(\log\theta)^{-2\beta}\,d\theta<\infty,
\]
which gives the stated criterion.
\end{proof}

We need one elementary rigidity fact for a subarc that is already $C^1$.

\begin{lemma}\label{lem:local-rigidity}
    Let $\gamma\colon [a,b] \to\bR^2$ be a Liouville regular Jordan arc and suppose that $\gamma|_J$ is a $C^1$-embedding on an open subinterval $J \subset (a,b)$.
    If $f$ is a Liouville primitive for $\gamma$, then for $s,t\in J$,
    \[
      f(t)-f(s)=\int_s^t\gamma^*\lambda.
    \]
\end{lemma}
\begin{proof}
    Choose an approximating sequence $(\gamma_n,f_n)_n$ for $(\gamma,f)$, and fix $s<t$ in $J$.  
    We show that
    \[
      f(t)-f(s)=\int_s^t\gamma^*\lambda.
    \]

    Let $\varepsilon_0>0$.  
    Since $\gamma|_{[s,t]}$ is a compact $C^1$-embedded arc, it has a sufficiently thin tubular neighborhood $Q_{\varepsilon_0}$ that is a curvilinear rectangle with the following properties:
    \begin{enumerate}[(i)]
    \item $\gamma|_{[s,t]}$ is a crosscut of $Q_{\varepsilon_0}$ joining its two transverse sides and otherwise contained in its interior;
    \item $\Area (Q_{\varepsilon_0})<\varepsilon_0$.
    \end{enumerate}
    Fix a $C^1$-diffeomorphism $\varphi \colon [-1,1]\times [s,t]\to Q_{\varepsilon_0} $ such that $\varphi |_{\{0\}\times [s,t]}=\gamma|_{[s,t]}$. 
    For any $\varepsilon \in (0,\varepsilon_0)$, we can choose $\delta_\varepsilon \in (0,1)$ satisfying $\delta_\varepsilon \searrow 0$ as $\varepsilon \searrow 0$ and $\Area (\varphi ([-\delta_\varepsilon, \delta_\varepsilon] \times [s,t]))<\varepsilon$.
    Set $Q_\varepsilon \coloneqq \varphi ([-\delta_\varepsilon, \delta_\varepsilon] \times [s,t])$ for each $0< \varepsilon <\varepsilon_0$.
    
    Fix a real number $\varepsilon\in (0, \varepsilon_0)$. 
    For all sufficiently large $n$, uniform convergence $\gamma_n\to\gamma$ implies that $\gamma_n$ has a subarc
    \[
      \eta_{\varepsilon,n}=\gamma_{n}|_{[s_{\varepsilon, n},t_{\varepsilon, n}]}
    \]
    that is also a crosscut of $Q_\varepsilon$ joining the same two transverse sides, where $s_{\varepsilon, n}\to s$ and $t_{\varepsilon, n}\to t$ as $n \to \infty$. 

    Write $\eta=\gamma|_{[s,t]}$. 
    Then $\varphi(\{\delta_\varepsilon\}\times [s,t])$ is one of the two long sides of $Q_\varepsilon$. 
    Complete each of $\eta$ and $\eta_{\varepsilon, n}$ to a Jordan curve by adjoining this long side and the appropriate portions of the two transverse sides. 
    Denote the resulting return arcs by $r_\varepsilon$ and $r_{\varepsilon, n}$, and the corresponding Jordan domains by $D_\varepsilon$ and $D_{\varepsilon, n}$, respectively.
    Both domains are contained in $Q_\varepsilon$. 
    Moreover, the endpoints of $\eta_{\varepsilon, n}$ converge to those of $\eta$, so $r_{\varepsilon}$ and $r_{\varepsilon,n}$ differ only by two subarcs of the transverse sides whose lengths tend to zero. 

    Consequently,
    \[
        \int_{r_{\varepsilon,n}}\lambda\to\int_{r_\varepsilon} \lambda \quad (n\to \infty).
    \]
    Applying Green's theorem to the two Jordan curves, with consistent orientations, gives
    \[
      \left|
        \int_{\eta_{\varepsilon,n}}\lambda-\int_\eta\lambda
      \right|
      \leq
      \left|
        \Area(D_{\varepsilon,n})-\Area(D_\varepsilon)
      \right|
      +
      \left|
        \int_{r_{\varepsilon,n}}\lambda-\int_{r_\varepsilon}\lambda
      \right|.
    \]
    Since $D_\varepsilon,D_{\varepsilon,n}\subset Q_\varepsilon$, we have
    \[
      \left|
        \Area (D_{\varepsilon,n})-\Area (D_\varepsilon)
      \right|
      \leq \Area (Q_\varepsilon)<\varepsilon.
    \]
    It follows that
    \begin{equation}\label{eq:limsup}
        \limsup_{n\to\infty}
      \left|
        \int_{s_{\varepsilon,n}}^{t_{\varepsilon,n}}\gamma_n^*\lambda
        -
        \int_s^t\gamma^*\lambda
      \right|
      \leq\varepsilon.
    \end{equation}

    Since $s_{\varepsilon,n}\to s$ and $t_{\varepsilon,n}\to t$,
    uniform convergence $f_n\to f$ and continuity of $f$ give
    \[
      \int_{s_{\varepsilon,n}}^{t_{\varepsilon,n}}\gamma_n^*\lambda
      = f_n(t_{\varepsilon,n})-f_n(s_{\varepsilon,n})
      \to f(t)-f(s) \quad (n\to\infty).
    \]
    Hence \eqref{eq:limsup} yields
    $\bigl|f(t)-f(s)-\int_s^t\gamma^*\lambda\bigr|\le\varepsilon$.
    Since $\varepsilon$ is arbitrary, the desired identity follows.
\end{proof}

\begin{theorem}\label{thm:spiral}
The spiral $\gamma_{\alpha,\beta}$ is Liouville regular if and only if
\[
  2\alpha>1,
  \quad\text{or}\quad
  2\alpha=1\ \text{ and }\ 2\beta>1.
\]
\end{theorem}
\begin{proof}
    Assume first that the improper integral in \cref{lem:spiral-action} converges. 
    Reparametrize the compactified spiral, without changing notation, by a finite interval $[a_{\alpha,\beta},R_\infty](\simeq [a_{\alpha,\beta},\infty])$ so that it is smooth on $[a_{\alpha,\beta},R_\infty)$. 
    Choose a strictly increasing sequence $(R_n)_n$ of real numbers greater than $a_{\alpha,\beta}$ with $R_n\nearrow R_\infty$ and orientation-preserving $C^1$-diffeomorphisms
    \[
      \varphi_n\colon
      [a_{\alpha,\beta},R_\infty]
      \to
      [a_{\alpha,\beta},R_n]
    \]
    that agree with the identity on $[a_{\alpha,\beta},R_{n-1}]$. 
    Set $\gamma_n=\gamma_{\alpha,\beta}\circ\varphi_n$.
    Then each $\gamma_n$ is a $C^1$-embedded arc, and uniform continuity of the compactified spiral gives $\gamma_n\to \gamma_{\alpha,\beta}$ uniformly.
    Define
    \[
      f(t) \coloneqq \int_{a_{\alpha,\beta}}^t
      \gamma_{\alpha,\beta}^*\lambda
      \quad (t\in [a_{\alpha,\beta},R_\infty)),
    \]
    and define $f(R_\infty)$ to be the corresponding improper integral.  
    Its convergence means precisely that $f$ is a continuous function defined on $[a_{\alpha,\beta},R_\infty]$.
    The primitive of $\gamma_n$ normalized at $a_{\alpha,\beta}$ is
    \[
      f_n(t)
      =
      \int_{a_{\alpha,\beta}}^t \gamma_n^*\lambda
      =
      f\bigl(\varphi_n(t)\bigr).
    \]
    Since $f$ is uniformly continuous and $\varphi_n\to\id$ uniformly, we have
    \[
      f_n=f\circ\varphi_n\to f
    \] uniformly on $[a_{\alpha,\beta},R_\infty]$.
    Hence, $\gamma_{\alpha,\beta}$ is Liouville regular.
    
    Conversely, suppose that the compactified spiral is Liouville regular with Liouville primitive $f$. 
    By \cref{lem:local-rigidity}, on every compact subinterval of the open spiral, the function $f$ differs by a constant from the usual primitive $\int\xi\,dx$.  
    If the improper integral diverges, this usual primitive has no finite continuous limit at the central endpoint.  
    This contradicts the continuity of~$f$.
\end{proof}

\begin{remark}
    Let $\gamma$ be a subarc of a Liouville regular Jordan curve whose image has measure zero.
    Then a Liouville primitive for $\gamma$ is unique.
    This follows from \cref{prop:closing}, whose proof gives the Liouville primitive as in \cref{prop:concat}, and the uniqueness of a Liouville primitive for a Liouville regular Jordan curve whose image has measure zero~\cite[Rem.~5.6]{AI_peg}. 
    The above assertion gives another proof of \cref{lem:local-rigidity}.
    We note that the proof of \cite[Rem.~5.6]{AI_peg} depends on microlocal sheaf theory.
\end{remark}

\subsection{A finite-coil non-example}\label{subsec:finite-coils}

The preceding examples have an actual infinite spiral at the singular endpoint. 
The next construction shows that this is not essential.  
We can arrange that every individual coil makes only finitely many turns, and that the curve is even differentiable at the accumulation point.  

We first record the elementary local construction that will be inserted at smaller and smaller scales.

\begin{lemma}\label{lem:finite-coil}
    There are universal constants $C,C'>0$ with the following property.
    For every $R>0$ and every integer $N\geq1$, there exists a smooth embedded arc $\sigma_{R,N}$ connecting $(-R,0)$ to $(R,0)$, agreeing with the $x$-axis near its endpoints, and otherwise contained in $B(0;R)$, such that
    \[
        \left|\int_{\sigma_{R,N}}\lambda\right|\leq C R^2.
    \]
    Moreover, there is a point $q_{R,N}$ on $\sigma_{R,N}$ such that the subarc $\sigma'_{R,N}$ of $\sigma_{R,N}$ connecting $(-R,0)$ and $q_{R,N}$ satisfies
    \[
        \left|\int_{\sigma'_{R,N}}\lambda\right|
        \geq C' N R^2-C R^2.
    \]
\end{lemma}
\begin{proof}
For $R>0$ and $N\geq1$, define
\[
  r(\theta)
  =
  R\left(
    \frac12-\frac{\theta}{8\pi N}
  \right),
  \quad
  0\leq\theta\leq2\pi N.
\]
Thus $r$ decreases linearly from $R/2$ to $R/4$.  Consider the inward spiral
\[
  \kappa(\theta)
  =
  \bigl(
    r(\theta)\cos(\theta+\pi),
    r(\theta)\sin(\theta+\pi)
  \bigr),
  \quad
  0\leq\theta\leq2\pi N.
\]
It is embedded, makes exactly $N$ turns, and is contained in
$\overline{B(0;R/2)}$.
Since replacing $\theta$ by $\theta+\pi$ changes the signs of both
coordinates, the pullback $\kappa^*\lambda$ is given by the same expression
as in the proof of \cref{lem:spiral-action}:
\[
  \kappa^*\lambda
  =
  -\frac12r(\theta)^2\,d\theta
  +
  \frac14d\bigl(r(\theta)^2\sin(2\theta)\bigr).
\]
Since $\sin(0)=\sin(4\pi N)=0$, the boundary term vanishes.
Using $r(\theta)\ge R/4$, we obtain
\begin{equation}
    \int_\kappa\lambda
  =
  -\frac12\int_0^{2\pi N}r(\theta)^2\,d\theta
  \leq
  -\frac{\pi}{16}NR^2.
  \label{eq:inward-coil-action}
\end{equation}
In particular,
\[
  \left|\int_\kappa\lambda\right|
  \geq
  \frac{\pi}{16}NR^2.
\]

We next construct the outward branch.  Put
\[
  \delta=\frac{R}{16N},
  \quad
  \widetilde r(\theta)=r(\theta)-\delta.
\]
During one complete turn, the radius of the first spiral decreases by
\[
  r(\theta)-r(\theta+2\pi)=\frac{R}{4N}.
\]
Since $\delta<R/(4N)$, the spiral with radius $\widetilde r(\theta)$
lies strictly between two successive turns of the first spiral.
Consequently, the two spirals are disjoint.  Traverse this second
spiral in the reverse direction, from $\theta=2\pi N$ to $\theta=0$,
so that it runs outward.

By a similar calculation, taking the reversed orientation into account, we find that the integral of $\lambda$ along the outward branch is
\[
  \frac12\int_0^{2\pi N}\widetilde r(\theta)^2\,d\theta.
\]
Therefore, the sum of the integrals of $\lambda$ along the two spiral branches is
\[
\begin{split}
  -\frac12\int_0^{2\pi N}r(\theta)^2\,d\theta
  +
  \frac12\int_0^{2\pi N}\widetilde r(\theta)^2\,d\theta
  &=
  -\frac12\int_0^{2\pi N}
  \bigl(r(\theta)^2-\widetilde r(\theta)^2\bigr)\,d\theta.
\end{split}
\]
Since
\[
  r^2-\widetilde r^2
  =
  2r\delta-\delta^2
  \leq 2r\delta,
\]
we obtain
\[
\begin{split}
  \left|
    \int_{\text{inward branch}}\lambda
    +
    \int_{\text{outward branch}}\lambda
  \right|
  &\leq
  \int_0^{2\pi N}r(\theta)\delta\,d\theta\\
  &\leq
  (2\pi N)\frac{R}{2}\frac{R}{16N}\\
  &=
  \frac{\pi}{16}R^2.
\end{split}
\]
Thus the factor $N$ cancels between the two branches, and the remaining error is bounded by a constant times $R^2$,
independently of $N$.  Geometrically, this remaining quantity is the
signed area of the thin ribbon between the two spirals.

Join the inner endpoints of the two spirals by a short embedded arc,
and join their outer endpoints to $(-R,0)$ and $(R,0)$ by embedded
connectors lying in $B(0;R)$ and disjoint from the spiral branches.
These connectors, together with the short arc joining the inner endpoints, can be chosen to have total length at most $C_0R$,
where $C_0$ is universal.  
Since $|\xi|\leq R$ in $B(0;R)$, every such
connector $\eta$ satisfies
\[
  \left|\int_\eta\lambda\right|
  =
  \left|\int_\eta\xi\,dx\right|
  \leq
  R\int_\eta|dx|
  \leq
  R\,\operatorname{length}(\eta).
\]
Hence the total contribution of the connectors is at most
$C_0R^2$. 
The finitely many corners can be smoothed in disjoint small
neighborhoods, preserving embeddedness and changing the integral of $\lambda$ by at most another universal multiple of $R^2$.  We thereby obtain a smooth
embedded arc $\sigma_{R,N}$ satisfying
\[
  \left|\int_{\sigma_{R,N}}\lambda\right|
  \leq CR^2
\]
for a universal constant $C$.

Finally, let $q_{R,N}$ be a point at the end of the inward spiral,
before the arc begins to return along the outward branch.  The part
from $(-R,0)$ to the beginning of the inward spiral consists only of
one of the above connectors and a local smoothing, and the absolute value of the integral of $\lambda$ along it is bounded by $C_1R^2$.  Combining this with
\eqref{eq:inward-coil-action}, we obtain
\[
\begin{split}
  \left|
    \int_{\sigma'_{R,N}}\lambda
  \right|
  &\geq
  \left|\int_\kappa\lambda\right|-C_1R^2\\
  &\geq
  \frac{\pi}{16}NR^2-C_1R^2.
\end{split}
\]
After renaming the universal constants, this is the required estimate.
\end{proof}

\begin{proposition}\label{prop:finite-coil-nonexample}
    There exists a Jordan arc $\gamma\colon[0,1]\to\bR^2$ that is $C^\infty$ on $(0,1]$ and differentiable at $0$, with
    \[
      \gamma(0)=(0,0),
      \quad
      \gamma'(0)=(1,0),
    \]
    but not Liouville regular.
\end{proposition}
\begin{proof}
    For $n \ge 1$, set
    \[
      a_n \coloneqq 2^{-n},
      \quad
      R_n \coloneqq 2^{-3n},
    \]
    and 
    \[
      D_n \coloneqq B\bigl((a_n,0);R_n\bigr).
    \]
    These disks are pairwise disjoint, since $R_n+R_{n+1}<a_n-a_{n+1}$. Inside $D_n$, replace the horizontal diameter by the translate of $\sigma_{R_n,N_n}$ from \cref{lem:finite-coil}, where
    \[
        N_n=n\,2^{6n}.
    \]
    We call this translate the \emph{$n$-th coil}.
    Outside the disks, keep the $x$-axis.
    Since each inserted arc agrees with the horizontal diameter near the boundary of its disk, the resulting image is a Jordan arc, smooth away from the origin.

    Parametrize the part inside $D_n$ on the interval $[a_n-R_n,a_n+R_n]$ and use the standard parametrization $t\mapsto(t,0)$ outside these intervals.
    The parametrizations of the inserted coils may be chosen smooth and equal to the standard one near the endpoints.  
    For $t\in[a_n-R_n,a_n+R_n]$, we have
    \[
        |\gamma(t)-(t,0)|\leq 2R_n.
    \]
    Since $R_n/a_n\to0$, it follows that
    \[
        \frac{\gamma(t)-\gamma(0)}{t}\to(1,0)
        \quad (t\searrow0).
    \]
    Thus $\gamma$ is differentiable at the origin with the stated derivative.
    
    We now examine the classical primitive on the punctured arc. 
    For $t>0$ define
    \[
        A(t) \coloneqq -\int_t^1\gamma^*\lambda.
    \]
    This is well-defined since the compact subarc $\gamma([t,1])$ meets only finitely many coils.  
    For a point $p$ of $\gamma((0,1])$, we write $A(p)$ for the value of $A$ at the parameter of $p$ by abuse of notation.
    Let $e_n$ be the exit point of the $n$-th coil in the direction from $0$ to $1$.  
    By \cref{lem:finite-coil}, the integral of $\lambda$ over the $n$-th coil is bounded by $CR_n^2$. 
    Hence
    \[
      \sum_n\left|\int_{\text{$n$-th coil}}\lambda\right|
      \leq C\sum_nR_n^2<\infty.
    \]
    Consequently the sequence $A(e_n)$ has a finite limit as $n\to\infty$; in fact the same is true along points of the $x$-axis tending to the origin.

    On the other hand, let $q_n$ be the point in the $n$-th coil supplied by \cref{lem:finite-coil}.  
    Comparing the primitive at $q_n$ with the primitive at the exit of the same coil gives
    \[
      |A(q_n)-A(e_n)|
      \geq C'N_nR_n^2-2CR_n^2.
    \]
    Our choices give
    \[
      N_nR_n^2=n,
    \]
    so the right-hand side tends to $+\infty$.  
    Since $(A(e_n))_n$ converges, we find that $A$ is unbounded near $0$.
    In particular, $A$ has no continuous extension to $0$. 
    
    Suppose, for contradiction, that $\gamma$ is Liouville regular with Liouville primitive $f$.
    Since $\gamma$ is a $C^1$-embedding on the connected interval $(0,1]$, \cref{lem:local-rigidity} implies, by applying it on overlapping compact subintervals, that $f-A$ is constant on $(0,1]$. 
    Therefore, $f$ is unbounded near $0$. 
    This contradicts the continuity of $f$ at $0$. 
\end{proof}

\begin{remark}
    The point of \cref{prop:finite-coil-nonexample} is that the obstruction cannot be detected by recording only the values of the primitive at the points $e_n$ in the proof. 
    The obstruction is that the primitive is unbounded near $0$.
    Hence uniform control of the primitive in \cref{def:arc} is stronger than convergence of the total signed areas of the successive small-scale pieces.
\end{remark}

\section{A positive-area example}\label{sec:positive}

In this section, we give an example of a Liouville regular curve with positive area.
Note that if a Jordan curve has positive area, then it inscribes a rectangle of any similarity class by Lebesgue's density theorem.

Classical constructions of positive-area Jordan arcs go back to Lebesgue and Osgood~\cite{Lebesgue03,Osgood}; see, for example, \cite{NV22}.  
We use a two-sided ternary version of the fat Cantor construction.  Besides producing a Liouville regular arc of positive area, it shows that a normalized Liouville primitive need not be unique.

\begin{remark}\label{rem:preceding-zero-area}
    Every curve or arc covered by \cref{cor:pvar}, \cref{rem:rough-path}, or \cref{cor:holder_bis} has zero planar Lebesgue measure.  
Indeed, if a Jordan arc $\gamma$ is of finite $p$-variation for some $p<2$, the standard interval-covering argument gives an estimate for Hausdorff dimension
\[
  \dim_{\mathrm H}\gamma([a,b])\leq p<2,
\]
and hence $\gamma([a,b])$ has planar Lebesgue measure zero. 
Moreover, the first level of a geometric $2$-rough path has vanishing $2$-variation; applying the same interval-covering argument at exponent $2$, with the mesh tending to zero, gives an estimate for Hausdorff measure
\[
  \cH^2\bigl(\gamma([a,b])\bigr)=0.
\]
If $\Omega$ is a bounded planar H\"older domain, then $\partial\Omega$ is null with respect to planar Lebesgue measure by~\cite[Cor.~4]{SmithStegenga}. 
The construction below shows that Liouville regular curves can have positive planar area. 
\end{remark}

\subsection{The ternary rectangle move}

If $p_0,\ldots,p_m\in\bR^2$, write $[p_0,\ldots,p_m]$ for the oriented polygonal arc through the listed points.  
First consider the unit square $Q=[0,1]^2$, with entrance $(0,0)$ and exit $(1,1)$, and the two complementary boundary arcs
\[
  A_Q=[(0,0),(1,0),(1,1)],
  \quad 
  B_Q=[(0,0),(0,1),(1,1)].
\]
Choose
\[
  0<a_1<b_1<a_2<b_2<1
\]
and delete the two vertical strips
\[
  G_1=(a_1,b_1)\times(0,1),
  \quad 
  G_2=(a_2,b_2)\times(0,1).
\]
The three remaining rectangles are
\[
  Q_1=[0,a_1]\times[0,1],\quad 
  Q_2=[b_1,a_2]\times[0,1],\quad 
  Q_3=[b_2,1]\times[0,1].
\]
Replace $A_Q$ and $B_Q$ by
\begin{align*}
 A'_Q={}&[(0,0),(a_1,0),(a_1,1),(b_1,1),(b_1,0),
             (a_2,0),(b_2,0),(1,0),(1,1)],\\
 B'_Q={}&[(0,0),(0,1),(a_1,1),(b_1,1),(a_2,1),
             (a_2,0),(b_2,0),(b_2,1),(1,1)].
\end{align*}
Thus the two new arcs coincide along the upper side of $G_1$ and the lower side of $G_2$. 
On each $Q_j$, they run along the two complementary boundary arcs from one corner to the opposite corner.  
Both $A'_Q$ and $B'_Q$ are embedded.

Since $\lambda=\xi\,dx$, direct calculation gives
\[
  \int_{A_Q}\lambda=0,
  \quad 
  \int_{B_Q}\lambda=1,
  \quad 
  \int_{A'_Q}\lambda=b_1-a_1,
  \quad 
  \int_{B'_Q}\lambda=1-(b_2-a_2).
\]
Consequently, 
\begin{align*}
 \int_{A'_Q}\lambda-\int_{A_Q}\lambda&=\Area(G_1),
 \\
 \int_{B'_Q}\lambda-\int_{B_Q}\lambda&=-\Area(G_2),
 \\
 \int_{B'_Q}\lambda-\int_{A'_Q}\lambda
   &=\sum_{j=1}^3\Area(Q_j).
\end{align*}

We shall also use the horizontal version of this move and the versions in which the entrance and exit are any prescribed pair of opposite corners.  
They are obtained from the model above by quarter-turns and reflections, interchanging the labels of the two boundary arcs if necessary. 
If two arcs have the same endpoints, the difference of the integrals of $\lambda$ along these arcs is the integral over a closed polygonal chain and hence is a signed area.  
Therefore, for every such version, the replacement arcs $A'_Q,B'_Q$ may be labeled so that
\begin{align}
 \left|\int_{A'_Q}\lambda-\int_{A_Q}\lambda\right|
   &\leq \Area(G_1\cup G_2),
 \label{eq:ternary-A-bound}\\
 \left|\int_{B'_Q}\lambda-\int_{B_Q}\lambda\right|
   &\leq \Area(G_1\cup G_2),
 \label{eq:ternary-B-bound}\\
 \int_{B'_Q}\lambda-\int_{A'_Q}\lambda
   &=\sum_{j=1}^3\Area(Q_j).
 \label{eq:ternary-difference-general}
\end{align}
The arcs coincide along one side of each deleted strip, whereas in every remaining rectangle their restrictions are again complementary boundary arcs satisfying the last identity. 
We refer to this operation as the \emph{ternary rectangle move}.

\subsection{Alternating vertical and horizontal refinements}

Start from
\[
  Q_0=[0,1]^2
\]
and the two polygonal arcs
\begin{align*}
 \alpha_0&=[(0,0),(1,0),(1,1)], \\ 
 \beta_0&=[(0,0),(0,1),(1,1)].
\end{align*}
Set
\[
  \varepsilon_n \coloneqq 2^{-n-4},
  \quad 
  s_n \coloneqq \frac{1-\varepsilon_n}{3}.
\]
At the $n$-th step, split every remaining rectangle into three equal rectangular slabs, separated by two strips whose widths in the splitting direction are each $\varepsilon_n/2$ times the corresponding side length.  For odd $n$ the strips are vertical, and for even $n$ they are horizontal.  Apply the ternary rectangle move to the restrictions of $\alpha_{n-1}$ and $\beta_{n-1}$ in each remaining rectangle.  Denote the resulting embedded polygonal arcs by $\alpha_n$ and $\beta_n$ and the family of the $3^n$ remaining rectangles by $\cQ_n$.
We also set $\cQ_0 \coloneqq \{ Q_0 \}$.

The first two refinements are shown in \cref{fig:ternary-two-steps}.

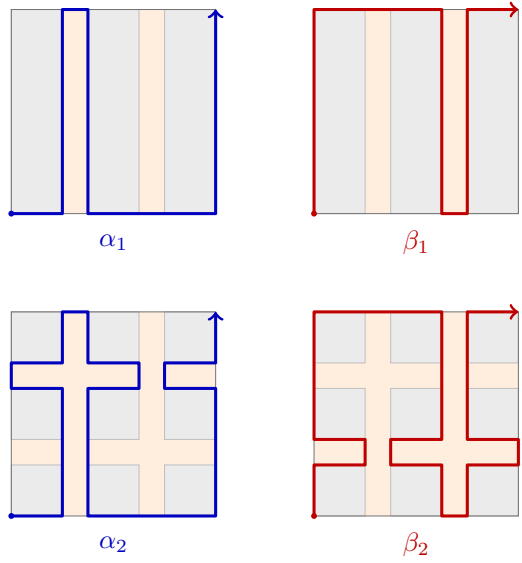
\begin{figure}[htbp]
\centering
\begin{tikzpicture}[
  x=1.35cm,
  y=1.35cm,
  active/.style={fill=black!8,draw=black!25,line width=0.2pt},
  removed/.style={fill=orange!13},
  frame/.style={draw=black!55,line width=0.35pt},
  alphapath/.style={
    draw=blue!75!black,
    line width=1.15pt,
    line join=round,
    line cap=round,
    ->
  },
  betapath/.style={
    draw=red!75!black,
    line width=1.15pt,
    line join=round,
    line cap=round,
    ->
  }
]
\begin{scope}[yshift=4cm]
  \fill[removed] (-1,-1) rectangle (1,1);
  \foreach \xa/\xb in {-1/-.5,-.25/.25,.5/1} {
    \path[active] (\xa,-1) rectangle (\xb,1);
  }
  \draw[frame] (-1,-1) rectangle (1,1);
  \draw[alphapath]
    (-1,-1) -- (-.5,-1) -- (-.5,1) -- (-.25,1)
    -- (-.25,-1) -- (.25,-1) -- (.5,-1) -- (1,-1) -- (1,1);
  \fill[blue!75!black] (-1,-1) circle (1.1pt);
  \node[font=\small,text=blue!75!black] at (0,-1.28) {$\alpha_1$};
\end{scope}

\begin{scope}[xshift=4cm,yshift=4cm]
  \fill[removed] (-1,-1) rectangle (1,1);
  \foreach \xa/\xb in {-1/-.5,-.25/.25,.5/1} {
    \path[active] (\xa,-1) rectangle (\xb,1);
  }
  \draw[frame] (-1,-1) rectangle (1,1);
  \draw[betapath]
    (-1,-1) -- (-1,1) -- (-.5,1) -- (-.25,1)
    -- (.25,1) -- (.25,-1) -- (.5,-1) -- (.5,1) -- (1,1);
  \fill[red!75!black] (-1,-1) circle (1.1pt);
  \node[font=\small,text=red!75!black] at (0,-1.28) {$\beta_1$};
\end{scope}

\begin{scope}
  \fill[removed] (-1,-1) rectangle (1,1);
  \foreach \xa/\xb in {-1/-.5,-.25/.25,.5/1} {
    \foreach \ya/\yb in {-1/-.5,-.25/.25,.5/1} {
      \path[active] (\xa,\ya) rectangle (\xb,\yb);
    }
  }
  \draw[frame] (-1,-1) rectangle (1,1);
  \draw[alphapath]
    (-1,-1) -- (-.5,-1) -- (-.5,-.5) -- (-.5,-.25)
    -- (-.5,.25) -- (-1,.25) -- (-1,.5) -- (-.5,.5)
    -- (-.5,1) -- (-.25,1)
    -- (-.25,.5) -- (.25,.5) -- (.25,.25) -- (-.25,.25)
    -- (-.25,-.25) -- (-.25,-.5) -- (-.25,-1) -- (.25,-1)
    -- (.5,-1)
    -- (1,-1) -- (1,-.5) -- (1,-.25) -- (1,.25)
    -- (.5,.25) -- (.5,.5) -- (1,.5) -- (1,1);
  \fill[blue!75!black] (-1,-1) circle (1.1pt);
  \node[font=\small,text=blue!75!black] at (0,-1.28) {$\alpha_2$};
\end{scope}

\begin{scope}[xshift=4cm]
  \fill[removed] (-1,-1) rectangle (1,1);
  \foreach \xa/\xb in {-1/-.5,-.25/.25,.5/1} {
    \foreach \ya/\yb in {-1/-.5,-.25/.25,.5/1} {
      \path[active] (\xa,\ya) rectangle (\xb,\yb);
    }
  }
  \draw[frame] (-1,-1) rectangle (1,1);
  \draw[betapath]
    (-1,-1) -- (-1,-.5) -- (-.5,-.5) -- (-.5,-.25)
    -- (-1,-.25) -- (-1,.25) -- (-1,.5) -- (-1,1)
    -- (-.5,1) -- (-.25,1)
    -- (.25,1) -- (.25,.5) -- (.25,.25) -- (.25,-.25)
    -- (-.25,-.25) -- (-.25,-.5) -- (.25,-.5) -- (.25,-1)
    -- (.5,-1)
    -- (.5,-.5) -- (1,-.5) -- (1,-.25) -- (.5,-.25)
    -- (.5,.25) -- (.5,.5) -- (.5,1) -- (1,1);
  \fill[red!75!black] (-1,-1) circle (1.1pt);
  \node[font=\small,text=red!75!black] at (0,-1.28) {$\beta_2$};
\end{scope}
\end{tikzpicture}
\caption{The first two refinements.  The widths of the deleted strips are
intentionally exaggerated.  The gray rectangles are the members of
$\cQ_n$, and the pale orange regions have been deleted by the displayed
stage.  
The blue and red arcs are $\alpha_n$ and $\beta_n$, respectively.}
\label{fig:ternary-two-steps}
\end{figure}

The parametrizations are chosen recursively. 
For $Q \in \cQ_n$, we define a closed interval $I_Q \subset [0,1]$ of length $5^{-n}$ as follows.
Define $I_{Q_0} \coloneqq [0,1]$.
For $Q \in \cQ_n$, divide $I_Q$ into five closed intervals of length $5^{-n-1}$.
The first, third, and fifth ones are assigned to the three children of $Q$ in $\cQ_{n+1}$, in the order in which $\alpha_{n+1}$ (equivalently $\beta_{n+1}$) passes through them.
The second and fourth ones are assigned to the two connecting segments in $Q$.
We set 
\[
    J_n \coloneqq \bigcup_{Q \in \cQ_n} I_Q.
\]
The intervals $I_Q$ for $Q \in \cQ_n$ are pairwise disjoint, and $J_{n+1} \subset J_n$.
On each $I_Q$, the arcs $\alpha_n$ and $\beta_n$ run along the two
complementary boundary arcs of $Q$, linearly on each edge.
The two arcs agree at the endpoints of $I_Q$.
On $[0,1] \setminus J_n$ they coincide, and they agree there with
$\alpha_m$ and $\beta_m$ for every $m\ge n$.

Set
\[
  V_n \coloneqq \bigcup_{Q\in\cQ_n}Q,
  \quad 
  V \coloneqq \bigcap_{n=0}^\infty V_n.
\]
At every step the fraction of area retained in each parent rectangle is $1-\varepsilon_n$.  Hence
\[
  \Area(V_n)=\prod_{k=1}^n(1-\varepsilon_k)
\]
and, by continuity from above of Lebesgue measure, 
\begin{equation}
  \Area(V)
  =\prod_{k=1}^\infty(1-\varepsilon_k)>0.
  \label{eq:fat-area}
\end{equation}
Here positivity follows from $\sum_k\varepsilon_k<\infty$.

Let
\[
  \delta_n \coloneqq \max_{Q\in\cQ_n}\diam(Q).
\]
Since the splitting direction alternates and $s_n<1/3$,
\begin{equation}
  \delta_n\leq \sqrt2\,3^{-\lfloor n/2\rfloor},
  \quad 
  \sum_{n=0}^\infty\delta_n<\infty.
  \label{eq:rectangle-diameter}
\end{equation}
The passage from level $n$ to level $n+1$ changes either arc only inside the rectangles of $\cQ_n$. 
Since $\alpha_{n+1}=\alpha_n$ and $\beta_{n+1}=\beta_n$ on $[0,1] \setminus J_n$, and all four maps send each $I_Q$ into $Q$ for $Q \in \cQ_n$, we have
\[
 \|\alpha_{n+1}-\alpha_n\|_\infty\leq\delta_n,
 \quad 
 \|\beta_{n+1}-\beta_n\|_\infty\leq\delta_n.
\]
Thus both sequences converge uniformly.  
Moreover, $\alpha_n$ and $\beta_n$ agree on $[0,1]\setminus J_n$ and map each $I_Q$ into $Q$ for $Q \in \cQ_n$, so
\[
  \|\alpha_n-\beta_n\|_\infty\leq\delta_n.
\]
Their limits are therefore the same map; denote it by
\[
  \gamma \colon[0,1]\to Q_0.
\]

The map $\gamma$ is injective. 
Indeed, let $t \neq t'$ and choose $n$ with $5^{-n} < |t-t'|$.
Then $t$ and $t'$ do not lie in the same interval $I_Q$ with $Q \in \cQ_n$.
At that stage they lie either in different remaining rectangles or in distinct connecting pieces. 
These sets are disjoint except at the common endpoint represented by a single parameter, and all later changes occur only inside the corresponding remaining rectangles.
Thus their limiting images are distinct.

The image of $\gamma$ contains $V$.  
To see this, let $P \in V$ and choose the nested rectangles $Q_n\in\cQ_n$ containing $P$. 
The corresponding nested intervals $I_{Q_n}$ have length $5^{-n}$ and hence determine a unique $t_P \in \bigcap_n I_{Q_n}$.
For every $n$, $\alpha_m$ and $\beta_m$ send $t_P$ into $Q_n$ for $m \ge n$.
Since $\diam(Q_n)\to0$, we obtain $\gamma(t_P)=P$. 
In view of \eqref{eq:fat-area}, $\gamma$ is a Jordan arc of positive measure.

\subsection{The two limiting primitives}

Normalize the polygonal primitives by
\[
 f_n^\alpha(t)=\int_0^t\alpha_n^*\lambda,
 \quad 
 f_n^\beta(t)=\int_0^t\beta_n^*\lambda.
\]
Let
\[
  E_{n+1}=\Area(V_n\setminus V_{n+1})
\]
be the area deleted in passing from level $n$ to level $n+1$.

\begin{lemma}\label{lem:ternary-action}
    There is a constant $C>0$, independent of $n$, such that
    \[
        \|f_{n+1}^\alpha-f_n^\alpha\|_\infty
        +\|f_{n+1}^\beta-f_n^\beta\|_\infty
        \leq C(E_{n+1}+\delta_n).
    \]
\end{lemma}
\begin{proof}
    If $t\notin J_n$, then $[0,t]$ is the union of $[0,t]\setminus J_n$, on which $\alpha_n=\alpha_{n+1}$, and the intervals $I_Q$ with $Q\in\cQ_n$ and $I_Q\subset[0,t]$.
    On each such $I_Q$, the arc $\alpha_{n+1}$ is the ternary replacement of $\alpha_n$.
    Hence by \eqref{eq:ternary-A-bound} and \eqref{eq:ternary-B-bound},
    \[
        |f_n^\alpha(t) - f_{n+1}^\alpha(t)| \le E_{n+1}.
    \]

    If $t \in I_Q$, the remaining contribution from $I_Q \cap [0,t]$ is bounded by $C\delta_n$.
    Indeed, both arcs consist of a fixed number of segments in $Q$, and $|\xi|\le 1$ on $Q_0$.
    
    The same argument applies to $\beta_n$ and $\beta_{n+1}$. 
\end{proof}

Since the sets $V_n$ are nested,
\[
  \sum_{n=0}^\infty E_{n+1}
  =\Area(V_0)-\Area(V)<\infty.
\]
Together with \eqref{eq:rectangle-diameter} and \cref{lem:ternary-action}, this shows that both sequences of primitives are uniformly Cauchy.  Write
\[
  f_n^\alpha\to f^\alpha,
  \quad 
  f_n^\beta\to f^\beta
\]
uniformly on $[0,1]$.

The two limits are not equal.  
On the connecting pieces the two level-$n$ arcs coincide, while on every $Q\in\cQ_n$ their two boundary routes satisfy \eqref{eq:ternary-difference-general}.  
Therefore
\begin{equation}
 f_n^\beta(1)-f_n^\alpha(1)=\Area(V_n).
 \label{eq:primitive-gap}
\end{equation}
Passing to the limit and using \eqref{eq:fat-area} gives
\[
  f^\beta(1)-f^\alpha(1)=\Area(V)>0,
\]
whereas $f^\alpha(0)=f^\beta(0)=0$.

\begin{proposition}\label{prop:positive-area}
    The common limit $\gamma \colon[0,1]\to\bR^2$ constructed above is a Liouville regular Jordan arc of positive planar Lebesgue measure.  Moreover, it admits two distinct Liouville primitives normalized to vanish at $0$.
\end{proposition}
\begin{proof}
    For each $n$, round the finitely many corners of $\alpha_n$ using \cref{lem:polygonal-rounding}, with both the $C^0$-error and the uniform primitive error smaller than $2^{-n}$.  
    The resulting $C^1$-embeddings converge uniformly to $\gamma$, and their primitives converge uniformly to $f^\alpha$. 
    Thus $\gamma$ is Liouville regular with Liouville primitive $f^\alpha$.  
    Applying the same argument to $\beta_n$ gives the Liouville primitive $f^\beta$.  
    They are distinct by \eqref{eq:primitive-gap}, and the positive-area assertion was proved above.
\end{proof}

\begin{corollary}\label{cor:positive-curve}
    There exists a Liouville regular Jordan curve of positive Lebesgue measure.
\end{corollary}
\begin{proof}
    Join $(1,1)$ to $(0,0)$ by a smooth embedded arc contained in the complement of $Q_0$ except at its endpoints. 
    This connecting arc is rectifiable, hence Liouville regular by \cref{cor:pvar}.  
    Its concatenation with $\gamma$ is a Jordan curve, so \cref{prop:closing} applies.  
    The resulting curve contains $V$ and therefore has positive measure.
\end{proof}

\printbibliography

\end{document}